\documentclass[times]{elsarticle}

\usepackage[colorlinks=true, pdfstartview=FitV, linkcolor=blue, pagebackref, citecolor=blue, urlcolor=blue]{hyperref}

\makeatletter
\providecommand{\doi}[1]{%
  \begingroup
    \let\bibinfo\@secondoftwo
    \urlstyle{rm}%
    \href{http://dx.doi.org/#1}{%
      doi:\discretionary{}{}{}%
      \nolinkurl{#1}%
    }%
  \endgroup
}
\makeatother
\usepackage{amssymb,amsmath}

\newcommand{\darkrad}{0.17}
\newcommand{\lrad}{0.4}

\usepackage{amsthm}
\numberwithin{equation}{section}

\newtheorem{theorem}[equation]{Theorem}

\newtheorem{prop}[equation]{Proposition}

\newtheorem{lemma}[equation]{Lemma}

\theoremstyle{remark}
\newtheorem*{rmk*}{Remark}

\newtheorem*{rmks*}{Remarks}
\newtheorem{rmks}[equation]{Remarks}

\theoremstyle{definition}

\DeclareMathOperator{\Ext}{Ext}

\DeclareMathOperator{\Hom}{Hom}
\DeclareMathOperator{\Lie}{Lie}

\DeclareMathOperator{\Aut}{Aut}

\DeclareMathOperator{\SO}{SO}
\DeclareMathOperator{\SU}{SU}
\DeclareMathOperator{\Sp}{Sp}
\DeclareMathOperator{\Spin}{Spin}
\DeclareMathOperator{\SL}{SL}
\DeclareMathOperator{\U}{U}
\DeclareMathOperator{\PU}{PU}
\DeclareMathOperator{\GL}{GL}
\DeclareMathOperator{\PGL}{PGL}

\DeclareMathOperator{\Res}{Res}

\DeclareMathOperator{\Sym}{Sym}

\DeclareMathOperator{\ad}{ad}

\newcommand{\Gm}{\mathbb{G}_{\mathrm{m}}}

\newcommand{\C}{{\mathbb C}}

\newcommand{\Z}{{\mathbb Z}}

\newcommand{\dE}{^2\!E_6}

\newcommand{\tD}{^3\!D_4}
\newcommand{\sD}{^6\!D_4}

\newcommand{\g}{\mathfrak{g}}

\newcommand{\so}{\mathfrak{so}}
\newcommand{\h}{\mathfrak{h}}
\newcommand{\m}{\mathfrak{m}}
\newcommand{\spin}{\mathfrak{spin}}
\newcommand{\su}{\mathfrak{su}}

\newcommand{\eps}{\epsilon}

\renewcommand{\j}{\mathfrak{j}}

\newcommand{\fu}{\mathfrak{u}}
\renewcommand{\sl}{\mathfrak{sl}}

\newcommand{\kbar}{\overline{k}}

\newcommand{\qform}[1]{\langle #1 \rangle}
\newcommand{\bilform}{\qform{\, ,}}

\title{Minuscule embeddings\tnoteref{t1}}
\tnotetext[t1]{In memory of T.A.~Springer}
\author[1]{Benedict H. Gross}
\ead[1]{gross@math.harvard.edu} 
\address[1]{Department of Mathematics, University of California, San Diego, 9500 Gilman Dr., La Jolla, CA 92093}

\author[2]{Skip Garibaldi}
\ead[2]{skip@garibaldibros.com}
\address[2]{IDA Center for Communications Research, La Jolla, 4320 Westerra Ct, San Diego, CA 92121}
\begin{document}

\begin{abstract}
We study embeddings $J \rightarrow G$ of simple linear algebraic groups with the following property: the simple components of the $J$ module $\Lie(G)/\Lie(J)$ are all minuscule representations of $J$. One family of examples occurs when the group $G$ has roots of two different lengths and $J$ is the subgroup generated by the long roots. We classify all such embeddings when $J = \SL_2$ and $J = \SL_3$, show how each embedding implies the existence of exceptional algebraic structures on the graded components of $\Lie(G)$, and relate properties of those structures to the existence of various twisted forms of $G$ with certain relative root systems.
\end{abstract}


\maketitle

\setcounter{tocdepth}{1}
\tableofcontents
% \newpage
\section{Introduction}

% SET THE LENGTH FOR THE DIAGRAMS!
\setlength{\unitlength}{.5cm}

\newsavebox{\Esev}
\savebox{\Esev}(7,1.9){\begin{picture}(7,1.9)
   % put in the circles
    \multiput(1,0.7)(1,0){6}{\circle*{\darkrad}}
     \put(3,1.45){\circle*{\darkrad}}
    % put in the lines
    \put(1,.7){\line(1,0){5}}
    \put(3,1.45){\line(0,-1){0.75}}
\end{picture}}

\newsavebox{\Evi}
\savebox{\Evi}(5,2){\begin{picture}(5,2)
    % put in the circles
    \multiput(0.5,0.5)(1,0){5}{\circle*{\darkrad}}
    \put(2.5,1.5){\circle*{\darkrad}}

    % put in the lines
    \put(0.5,0.5){\line(1,0){4}}
    \put(2.5,1.5){\line(0,-1){1}}\end{picture}}
        
\newsavebox{\Eviii}
\savebox{\Eviii}(8,1.9){\begin{picture}(8,1.9)
   % put in the circles
    \multiput(1,0.7)(1,0){7}{\circle*{\darkrad}}
     \put(3,1.45){\circle*{\darkrad}}
    % put in the lines
    \put(1,.7){\line(1,0){6}}
    \put(3,1.45){\line(0,-1){0.75}}
\end{picture}}

\newsavebox{\dEpic}
\savebox{\dEpic}(3.5,1){\begin{picture}(3.5, 1)
    % put in the vertices
    \multiput(1.75,0.75)(0.75,0){2}{\circle*{\darkrad}}
    \multiput(1.75,0.25)(0.75,0){2}{\circle*{\darkrad}}
    \multiput(0.25,0.5)(0.75,0){2}{\circle*{\darkrad}}
    
    % put in the lines
    \put(0.25, 0.5){\line(1,0){0.75}}
    \put(1.75,0.75){\line(1,0){0.75}}
    \put(1.75, 0.25){\line(1,0){0.75}}
    
    % put in the ovals
    \put(1.75,0.5){\oval(1.5,0.5)[l]}\end{picture}}
\newsavebox{\iEpic}
\savebox{\iEpic}(2.5,1){\begin{picture}(2.5,1)
    % put in the circles
    \multiput(0.25,0.25)(0.5,0){5}{\circle*{\darkrad}}
    \put(1.25,0.75){\circle*{\darkrad}}

    % put in the lines
    \put(0.25,0.25){\line(1,0){2}}
    \put(1.25,0.75){\line(0,-1){0.5}}
    
    \put(0.25,0.25){\circle{\lrad}}
    \put(2.25,0.25){\circle{\lrad}}
\end{picture}}

\newsavebox{\Fiv}
\savebox{\Fiv}(4,1){\begin{picture}(4,1)\multiput(0.5,0.5)(1,0){2}{\circle*{\darkrad}}
\multiput(2.5,0.5)(1,0){2}{\circle*{\darkrad}}
    % put in the lines
    \put(1.5,0.45){\line(1,0){1}}
    \put(1.5,0.55){\line(1,0){1}}
    \put(0.5,0.5){\line(1,0){1}}
    \put(2.5,0.5){\line(1,0){1}}
        \put(1.75, .35){\makebox(0.2,0.3)[s]{$>$}}\end{picture}}

\newsavebox{\Gii}
\savebox{\Gii}(4,1){\begin{picture}(4,1)
\multiput(0.5,0.5)(1.5,0){2}{\circle*{\darkrad}}
\put(0.5,0.44){\line(1,0){1.5}}
\put(0.5,0.5){\line(1,0){1.5}}
\put(0.5,0.56){\line(1,0){1.5}}
\put(1,0.35){{\small\mbox{$>$}}}\end{picture}}

In this paper we study embeddings $J \to G$ of simple linear algebraic groups over a field such that the simple factors of the composition series of the $J$-module $\Lie(G) / \Lie(J)$ are all minuscule representations of $J$.  We call such embeddings \emph{minuscule}.  

Recall that minuscule representations of a split, simple group $J$ over a field of characteristic zero are the irreducible representations whose weights for a maximal split torus lie in a single orbit for the Weyl group. (Unlike Bourbaki \cite [Ch VI, \S1, Ex 24] {Bou:g4} or \cite[Ch VIII, \S7.3] {Bou:g7}, we consider the trivial representation to be minuscule.) Over a general field, they are the irreducible representations whose highest weight is minimal for the partial ordering on the set of dominant weights (given by $\lambda \ge \mu$ if $\lambda - \mu$ is a sum of positive roots). Alternatively, they are the irreducible representations whose weights $\lambda$  satisfy $\langle \lambda, \alpha^{\vee} \rangle \in \{0,1,-1\}$ for all roots $\alpha$ \cite[Ch VIII, \S7.3, Prop 6]{Bou:g7}. Each minuscule representation is determined by its central character, and the number of minuscule representations is equal to the order of the finite center $Z(J)$. For the group $J = \SL_2$ only the trivial and the standard two dimensional representation are minuscule.

Minuscule embeddings arise naturally in several different contexts:
\begin{itemize}
\item When $G$ is a split group which has two root lengths and $J$ is the subgroup generated by the long roots.  Indeed, for a short root $\alpha$ and a long root $\beta$, the pairing $\qform{\alpha, \beta^\vee}$ is in $\{ 0, \pm 1 \}$.  See section \ref{two.lengths}.
\item When $J$ is the $A_1$ subgroup generated by the highest root of $G$ \cite[Ch VI, \S1]{Bou:g4} and its negative, as in \cite[Prop.~3.3]{GrWall}.  This is up to conjugacy the unique $A_1$ subgroup of $G$ with Dynkin index 1 \cite[Th.~2.4]{Dynk:ssub}.  See section \ref{A1.basic}.
\item Several rows of the Magic Triangle in \cite{DeligneGross} can be viewed in terms of minuscule embeddings, where $J = \SL_2$ or $\SL_3$ and 
$G$ is exceptional of type $E$, $F$, or $G$.  See sections \ref{A1.basic} and \ref{A2.basic}.
\item When $G$ has a relative root system with two root lengths such that the long roots have multiplicity 1 in $\Lie(G)$ and form the root system of $J$, see sections \ref{SL2.eg}, \ref{two.lengths}, \ref{SL3.eg}, and \ref{D4.eg}. 
\end{itemize}

%Recall that the minuscule representations of a split semi-simple group $J$ over a field of characteristic $0$ are the irreducible linear representations whose weights for a maximal torus lie in a single orbit for the Weyl group.  (In this definition, the trivial representation of $J$ is minuscule and 0 is a minuscule weight.)  Over a general field $k$, they are the irreducible representations whose highest weights are minimal for the partial ordering on the set of dominant weights \cite{Bou:g4}. Minuscule representations are determined by their central characters, and the number of minuscule representations is equal to the order of the finite center $Z(J)$. For the group $\SL_2$, only the standard and trivial representations are minuscule.

This paper includes a classification of the minuscule embeddings $J \rightarrow G$ over $k$, for $J = \SL_2$, $\SL_3$, and $\Spin_{4,4}$ (the split simply-connected group of type $D_4$).  We will assume, throughout this paper, that the characteristic of $k$ is not equal to $2$ or $3$, so that in particular Proposition \ref{decomp.prop} applies. Much of our work involves the study of the centralizer $Z_G(J)$ and its representations $W_\chi$, which are defined in the next section. These representations have exceptional invariant tensors, which were studied in detail by T.A. Springer \cite{Sp:jord}, \cite{Sp:ex}, \cite[\S38]{KMRT}, and it is a pleasure to dedicate this paper to his memory. We leverage knowledge of those tensor structures to give criteria for the existence of algebraic groups with relative root systems of type $BC_1$, $G_2$, and $F_4$. 

Regarding related work: After we had written this paper, we learned from Alberto Elduque of Vinberg's paper \cite{Vinberg:na}, where what we call a minuscule embedding $\SL_3 \to G$ is studied as a ``short $\SL_3$-structure on $\g$''.  Sections \ref{A2.basic} and \ref{SL3.invt} have substantial overlap with \cite{Vinberg:na}; one could view this material as a perspective on Springer's monograph \cite{Sp:jord}.  In another direction, 
the recent paper \cite{AF:shortP} begins with an isotropic semisimple group $G$ and also deduces algebraic structures on some  subspaces of $\Lie(G)$.  In yet another direction, the papers \cite{BermanMoody} and \cite{BenkartZelmanov} study embeddings $J \to G$ such that the nonzero weights of the representation $\g/\j$ of $J$ are all roots of $J$.

\section{Generalities}

A minuscule embedding $J \rightarrow G$  gives a grading of the Lie algebra $\g$ of $G$ over $k$, where the summands are indexed by the characters of the center $Z(J)$ of $J$. Since the center is a finite group scheme of multiplicative type over $k$, its Cartier dual $C = \Hom(Z(J),\mathbb G_m)$ is a finite \'etale group scheme. For each character $\chi$ in $C$ we let $V_\chi$ be the minuscule representation of $J$ whose weights restrict to $\chi$ on $Z(J)$.  If $\chi \ne \chi'$, then the difference of the highest weights of $V_\chi$ and $V_{\chi'}$ is not in the root lattice, so $\Ext^1_J(V_\chi, V_{\chi'}) = 0$ \cite[II.2.14]{Jantzen} and we have a direct sum decomposition
\begin{equation} \label{min.decomp}
\g/\j = \bigoplus_{\chi \in C} V_\chi \otimes W_\chi
\end{equation}
as representations of $J$. We note that each vector space $W_\chi$ is a linear representation of the centralizer $Z_G(J)$ of $J$ in $G$. For the minuscule embeddings which correspond to the long root subgroups of a split adjoint group with two root lengths, the centralizer $Z_G(J)$ is equal to the finite center $Z(J)$, $W_0 = 0$ and for each nonzero $\chi$, $W_\chi = \chi$.

We can make this decomposition more uniform by considering the Vinberg grading of $\g$ given by the action of the finite group scheme $Z(J)$. For nonzero $\chi$, the component $\g(\chi)$ is the representation $V_\chi \otimes W_\chi$ of the centralizer $G(0) = J.Z_G(J)$. For $\chi = 0$ the component $\g(0)$ is the Lie algebra of $G(0)$.

Let $H$ be the connected component of the centralizer $Z_G(J)$ and let $\h = \Lie(H)$. Since the intersection of $J$ and $Z_G(J)$ in $G(0)$ is the finite center $Z(J)$, if we assume that the characteristic of $k$ does not divide the order of $Z(J)$, the Lie algebra $\g(0) = \j + \h$ decomposes as a direct sum. In summary:

\begin{prop} \label{decomp.prop}
Assume that $J \rightarrow G$ is a minuscule embedding and that the finite group scheme $Z(J)$ has order prime to the characteristic of $k$.  Then we have the decomposition
\begin{equation} \label{sum.master}
\g =( \j \otimes 1) \oplus  (1 \otimes \h)  \oplus  \bigoplus_{0 \ne \chi \in C} (V_\chi \otimes W_\chi)
\end{equation}
as representations of $J \times Z_G(J)$.
\end{prop}

\section{$A_1$ case: Minuscule embeddings of $\SL_2$} \label{A1.basic}

Let $G$ be a split, simple group of adjoint type over $k$, of rank at least two. In this section, we will construct a minuscule embedding $\SL_2 \rightarrow G$ (generalizing the one studied over $\C$ in \cite{GrWall}), and will show that all such embeddings are conjugate.

The construction of a minuscule embedding of $\SL_2$ is given as follows. Let $T \subset B \subset G$ be a maximal torus contained in a Borel subgroup of $G$, and let $\beta$ be the highest root, which is the highest weight of $T$ on the adjoint representation $\g$. The 1-dimensional weight spaces $\g_{\beta}$ and $\g_{-\beta}$ generate a 3-dimensional Lie subalgebra of $\g$, which is isomorphic to $\sl_2$. A fixed embedding of $\SL_2$ sends the standard generators $E$ and $F$ of $\sl_2$ to compatible basis elements of $\g_{\beta}$ and of $\g_{-\beta}$ respectively. This embedding is minuscule. Indeed, a maximal torus $S$ in $\SL_2$ is the image of the co-root $\beta^{\vee}$, and for any positive root $\alpha$ which is not equal to $\beta$ we have $\langle \beta^{\vee},\alpha \rangle = 0$ or $\langle \beta^{\vee},\alpha \rangle = 1$. Hence the only representations of $\SL_2$ which occur in the quotient $\g/\sl_2$ are the standard and the trivial representation. 

\begin{theorem}
Every minuscule embedding $\SL_2 \rightarrow G$ is conjugate to the embedding given above.
\end{theorem}

\begin{proof}
If we have an embedding of $\SL_2$, then we may conjugate it by an element of $G$ so that the restriction to a maximal torus $S$ of $\SL_2$ lies in $T$, and is a dominant co-character $\nu$ with respect to $B$. Since the embedding is minuscule, for all positive roots $\alpha$, we have $\langle \nu,\alpha \rangle = 0,1,2$, and there is a unique positive root such that $\langle \nu,\alpha \rangle = 2$. Since the multiplicity of each simple root in $\alpha$ is less than or equal to its multiplicity in the highest root $\beta$, we must have $\langle \nu,\beta \rangle = 2$. Then the sub Lie algebra $\sl_2$ is given by $\g_{-\beta} + \Lie(S) + \g_{\beta}$ and $\nu = \beta^{\vee}$ is the associated co-root. We have therefore conjugated any embedding to have the same image as our standard embedding with equality on the maximal torus $S$. To finish the proof, we observe that the centralizer of $S$ acts transitively on the basis elements in the one dimensional $k$-vector space $\g_{\beta}$. Indeed, the centralizer of $S$ contains the maximal torus $T$. Since $G$ is adjoint and the root $\beta$ can be extended to give a root basis of the character group of $T$, there is a co-character $\mu:\mathbb G_m \rightarrow T$ which satisfies $\langle \mu,\beta \rangle = 1$.
\end{proof}

That is, every minuscule embedding $\SL_2 \rightarrow G$ is up to conjugacy the unique $A_1$ subgroup of $G$ with Dynkin index 1 \cite[Th.~2.4]{Dynk:ssub}.

For a fixed minuscule embedding $\SL_2 \rightarrow G$, we wish to determine the centralizer $H$ in $G$ and the full stabilizer $M$ in $\Aut(G)$. The calculation of the centralizer $H$ follows the argument in \cite[\S2]{GrWall}, but to determine the structure of the full stabilizer (which has connected component $H$) we need to consider the action of outer automorphisms of $G$. Fix a pinning of the simple root spaces with respect to $B$ and let $\Sigma$ be the group of all pinned automorphisms of $G$. This is a finite group, which is trivial unless $G$ is of type $A_n$ with $n \geq 2$, $D_n$ with $n \geq 4$, or $E_6$. In all but one of these cases, the group $\Sigma$ has order $2$. When $G$ has type $D_4$, the group $\Sigma$ has order $6$ and is isomorphic to the permutation group on $3$ letters.  The group $\Sigma$ permutes the simple roots, via the automorphisms of the Dynkin diagram. Since the multiplicity of a simple root in the highest root $\beta$ depends only on its orbit under $\Sigma$, the group $\Sigma$ fixes the highest root. Hence $\Sigma$ acts on the highest root space $\g_{\beta}$. In all cases but type $A_{2n}$, the group $\Sigma$ acts trivially on $\g_{\beta}$, whereas in the case of $A_{2n}$ it acts by the non-trivial character. This follows from the following more general result.

\begin{lemma} \label{A2n.lem}
Let $\Sigma$ be the group of all pinned automorphisms of $G$, and let $\alpha$ be a root fixed by $\Sigma$. Then $\Sigma$ acts trivially on the root space $\g_{\alpha}$, except in the case when $G$ has type $A_{2n}$, where $\Sigma$ acts on $\g_{\alpha}$ by the sign character.
\end{lemma}

\begin{proof}
We compute the trace of each non-trivial element $\sigma$ in $\Sigma$ in two ways. The first uses the grading of $\g$ into eigenspaces for $\sigma$. When $\sigma$ has order two, it suffices to determine the dimension of the fixed algebra. For $\g = \sl_{2n}$ the fixed algebra is $\mathfrak {sp}_{2n}$ and the trace of $\sigma$ is $2n+1$. For $\g = \sl_{2n+1}$ the fixed algebra is $\so_{2n+1}$ and the trace of $\sigma$ is $-2n$. For $\g = \so_{2n}$ the fixed algebra is $\so_{2n-1}$ and the trace of $\sigma$ is $2n^2 - 5n + 2$. Finally, for $\g = \mathfrak e_6$ the fixed algebra is $\mathfrak f_4$ and the trace of $\sigma$ is $26$. For $\sigma$ of order $3$ acting on $\so_8$, the trace of $\sigma$ is $7$.  (The fixed algebras are determined in \cite[Exercise VIII.5.13]{Bou:g7}, for example.)

We can also compute the trace of $\sigma$ using the Cartan decomposition $\g = \mathfrak{t} + \sum_{\alpha} \g_{\alpha}$. The trace on $\mathfrak {t}$ can be computed by comparing the rank with the rank of the fixed algebra. The only root spaces that contribute to the trace are those fixed by $\sigma$, and a count of the fixed roots shows that the trace of $\sigma$ on each of these spaces must be $+1$, except in the case of $A_{2n}$, when it must be $-1$. 
\end{proof}

The full automorphism group of $G$ is isomorphic to the semi-direct product $G.\Sigma$. This acts transitively on the set of minuscule embeddings $\SL_2 \rightarrow G$, and the stabilizer $M$ of our fixed embedding is an extension 
\begin{equation} \label{etale.seq}
1 \rightarrow H \rightarrow M \rightarrow \Sigma \rightarrow 1.
\end{equation}
When $G$ is not of type $A_{2n}$ this extension is split. Indeed, the group $\Sigma$ fixes the minuscule embedding described above. We shall see that it is not split for type $A_{2n}$. \\

\renewcommand*{\arraystretch}{1.4}
\begin{table}[hbt]
\[
\begin{array}{ccc}
G&M&W \\ \hline
\PGL_{n+2}&\GL_n.2&V_n + V_n^{\vee} \\
\SO_{2n+5}& \SL_2 \times \SO_{2n+1}&V_2 \otimes V_{2n+1} \\
\Sp_{2n+2}/\mu_2&\Sp_{2n}& V_{2n} \\
\SO_{2n+4}/\mu_2& (\SL_2 \times O_{2n})/\Delta \mu_2&V_2 \otimes V_{2n}  \\
\SO_8/\mu_2 & (\SL_2^3/{\prod \mu_2 = 1}).S_3 & V_2 \otimes V_2 \otimes V_2\\
G_2& \SL_2& \Sym^3(V_2) = V_4 \\
F_4& \Sp_6& \wedge^3(V_6)_0 = V_{14} \\
E_6/\mu_3  & (\SL_6/\mu_3).2& \wedge^3(V_6) = V_{20}  \\
E_7/\mu_2  & \Spin_{12}/\mu_2& V_{32} \text{\ (half-spin)} \\
E_8 & E_7 & V_{56} \text{\ (minuscule)} 
\end{array}
\]
\caption{For a minuscule $\SL_2$ in $G$, the group $M$ and its representation $W$} \label{BC1.table}
\end{table}

 The action of $\SL_2 \times M$ on $\g$ decomposes as in \eqref{sum.master} as a direct sum of representations
\begin{equation} \label{sum.sl2}
\g = \sl_2 \otimes 1 +  1 \otimes \m + V_2 \otimes W
\end{equation}
where $\m$ is the adjoint representation of (the disconnected reductive group) $M$. The center of $M$ is isomorphic to $\mu_2$ and the map $(\SL_2 \times M) \rightarrow \Aut(G)$ has kernel the diagonally embedded $\mu_2$.\\

Table \ref{BC1.table} lists the groups $M$  of automorphisms of $G$ which fix the minuscule embedding and their irreducible symplectic representations $W$. The connected component of $M$ is the centralizer $H$ of the embedding in $G$. For $G = \PGL_{n+2}$, $H = \GL_n$ is the Levi subgroup of a Siegel parabolic in $\Sp(W) = \Sp_{2n}$ and $M$ is its normalizer. This is a semi-direct product when $n$ is even, by Lemma \ref{A2n.lem}. When $n$ is odd, the exact sequence $1 \rightarrow H \rightarrow M \rightarrow \Z/2\Z \rightarrow 1$ is not split -- the smallest order of an element in the normalizer which does not lie in $H$ is $4$.

%%%%%%%%%%%%%%%%%%%%%%%%%%%%%%%%%
\section{$A_1$ case: $M$-Invariant tensors on $W$} \label{SL2.invt}

Fix a minuscule embedding $\SL_2 \rightarrow G$ associated to the highest root $\beta$.  The co-character $\beta^{\vee}$ gives a $5$-term grading on $\g$:
\begin{equation} \label{5term}
\g = \g_{-2} \oplus \g_{-1} \oplus \g_0 \oplus \g_1 \oplus \g_2.
\end{equation}
Each summand is a representation of $M = H.\Sigma$, which fixes the minuscule embedding. The subalgebra $\g_0$ is the Lie algebra of the reductive subgroup $H.S$ of $G$, and the eigencomponents $\g_2$ and $\g_{-2}$ are the highest and lowest weight spaces for the torus. Both have dimension $1$ with a chosen basis element (the images of the elements $E$ and $F$ in $\sl_2$) and give the trivial representation of $M$.  Let $W = \g_1$. Then $W$ is an irreducible representation of $M$. The Lie bracket $\wedge^2\g_1 \rightarrow \g_2$ gives a non-degenerate alternating bilinear form $\bilform$ on $W$  which is $M$-invariant via
\begin{equation} \label{bil.def}
[w,w'] = \langle w, w' \rangle E \quad \text{for $w, w' \in W$},
\end{equation}
 so $W$ is a symplectic representation of $M$.\\

%The summand $\g_{-1}$ is also isomorphic to $W$ as a representation of $M$.
We have already defined an $M$-invariant alternating bilinear form $\bilform$ on $W$ in \eqref{bil.def}. Using the chosen basis element (which is the image of $F$) of $\g_{-2}$ we can define an $M$-invariant quartic form $q$ on $W$ by the formula
\[
(\ad w)^4 \, F = q(w) E \quad \text{for $w \in W$.}
\]

For $G$ not of type $A_n$, there is a unique simple root $\gamma$ that is not orthogonal to $\beta$ and $W$ is, as a subspace of $\g$, a sum of the root subalgebras $\g_\alpha$ for $\alpha$ such that, when written as a sum of simple roots, the coefficient of $\gamma$ is 1.  By \cite[Th.~2f]{ABS}, there is an open orbit in $W$ under $H.T$, equivalently, under the group generated by $H$ and the image of the coroot $\beta^\vee$.  As $\beta^\vee$ acts by scalars on $W$, we find that there is an open $H$-orbit in $\mathbb{P}(W)$, whence $k[W]^H = k[f]$ for a (possibly constant) homogeneous $f$.

When $G$ has type $C_n$, $M = \Sp(W)$.  Because the nonzero vectors in $W$ are a single $\Sp(W)$-orbit, this representation has no invariant symmetric tensors of degree greater than zero, and in particular $q = 0$. In all other cases, $q$ is a non-zero quartic that generates the ring of $M$-invariant polynomials on $W$. Note that in the case when $G$ has type $A_n$ the subgroup $H$ fixes a quadratic form $q_2$ on $W$ \cite [Prop 6.1]{GrWall}. However, the form $q_2$ is not $M$-invariant: the quotient $M/H$ acts non-trivially and the first non-trivial invariant is the quartic $q = q_2^2$.\\

For types $B$, $D$, and $E$, it is a theorem \cite[Th.~27]{Helenius} that $q$ and $\bilform$ satisfy the algebraic identities defining a Freudenthal triple system as in \cite{Brown:E7}, \cite{Meyb:FT}, or \cite{Sp:e7}.  In the simplest case, when $G$ is split of type $D_4$, $M = (\SL_2 \times \SL_2 \times \SL_2 / {\prod \mu_2 = 1}).S_3$, and $W = V_2 \otimes V_2 \otimes V_2$ is the tensor product of the natural two dimensional representations.  The quartic form $q$ is Cayley's hyperdeterminant from \cite{Cayley:hyper}.  In another simple case, when $G$ is split of type $E_6$, $H$ is $(\SL_6/\mu_3).2$ and $W = \wedge^3 k^6$.  The quartic form is described in \cite[p.~83]{SK} or \cite[p.~4773]{BGL}.  When $G$ is split of type $E_8$, $q$ is the famous $E_7$ quartic in 56 variables described in, for example, \cite{Frd:E7} and \cite{Lurie}.

\begin{theorem}
For $G$, $M$, and $W$ as in Table \ref{BC1.table},
$M$ is the subgroup of $\GL(W)$ that stabilizes the two tensors $\bilform$ and $q$. 
\end{theorem}

\begin{proof}
This is clear for type $C_n$, where $q=0$ and $M = \Sp(W)$ is the subgroup of $\GL(W)$ stabilizing the non-degenerate symplectic form. For type $A_n$, the stabilizer of $\bilform$ and the quadratic form $q_2$ is the Levi subgroup $H$ of a Siegel parabolic in $\Sp(W)$, and the stabilizer of the quartic form $q = q_2^2$ is its normalizer $M = H.2$. In the remaining cases, the stabilizer of $q$ in $\GL(W)$ has been determined, for example, in \cite[\S9]{BGL} and for $G$ of type $E_8$ in \cite{Sp:e7}. This is the subgroup $\mu_4.M$; the subgroup $\mu_2.M = M$ also stabilizes the bilinear form $\bilform$.
\end{proof}

%%%%%%%%%%%%%%%%%%%%%%%%%%%%%%%%%

\section{$A_1$ case: Twisting and tensor structures}

We now change our notation and let $G$ be a simple group of adjoint type over $k$ with a minuscule embedding $\SL_2 \rightarrow G$. (We use $G_0$ to denote its split form, which is the group studied in the previous sections.) For example, suppose that the group $G$ has a relative root system of type $BC_1$ over $k$. Such a $G$ has a maximal split torus $S \cong \mathbb G_m$ of dimension one, whose non-trivial characters on $\g$ are $\{\pm1, \pm2 \}$. If we assume further that the long root spaces $\g_2$ and $\g_{-2}$ have dimension one, then $\g_{-2} + \Lie(S) + \g_{2}$ is a Lie subalgebra isomorphic to $\sl_2$. If we fix this isomorphism, the corresponding embedding $\SL_2 \rightarrow G$ is minuscule. We will want to identify these groups of rank one with certain tensor structures over $k$.\\

Choose an isomorphism
$$\phi: G_0 \rightarrow G.$$
of algebraic groups over the separable closure $\overline{k}$. Then for every element $\sigma$ in the Galois group of $\overline{k}$ over $k$, the composition
$$a(\sigma) = \phi^{-1} \circ  \sigma(\phi)$$
is an automorphism of $G_0$ over $\overline{k}$. This gives a 1-cocycle on the Galois group of $\overline{k}$ over $k$ with values in $\Aut(G_0)(\overline{k})$ . Since the minuscule embeddings of $\SL_2$ into $G_0$ form a single orbit for the automorphism group, we may modify our chosen isomorphism $\phi$ by an automorphism of $G_0$ so that it induces the identity map on the embedded subgroup $\SL_2$ in $G_0$ and $G$. Then $a(\sigma)$ lies in the stabilizer $M_0$ of the minuscule embedding and defines a $1$-cocycle on the Galois group with values in $M_0(\overline{k})$. The image of this cocycle under the map $M_0 \to \Aut(W_0, \bilform_0,q_0)$ determines a pure form $M$ of the stabilizer $M_0$ over $k$, or equivalently, a form $W$ of the tensor structure we have studied on $W_0$. The image under the map $H^1(k,M_0) \rightarrow H^1(k,\Aut(G_0))$ determines the isomorphism class of the twisted group $G$, and the twisted representation $W$ of $M$ occurs in the decomposition of its Lie algebra as in \eqref{sum.sl2}.

The map $M_0 \xrightarrow{\iota} \Sigma$ in \eqref{etale.seq} sends the 1-cocycle $a$ to a 1-cocycle $\iota(a)$ with values in $\Sigma$.  Recalling that $\Sigma$ is isomorphic to the symmetric group on $d$ letters for some $d$, $\iota(a)$ determines a degree $d$ \'etale $k$-algebra $K$ up to $k$-algebra isomorphism.  The groups $G$ and $H$ are of inner type if and only if $a$ is in the image of the map $H^1(k, H_0) \to H^1(k, M_0)$, equivalently, if and only $K$ is ``split'', i.e., is isomorphic to a product of copies of $k$.

In case $K$ is \emph{not} split, we twist sequence \eqref{etale.seq} by the 1-cocycle $\iota(a)$ to obtain an exact sequence of group schemes
\[
1 \to H_q \to M_q \to \Sigma_q \to 1.
\]
Here, $H_q$ is a quasi-split form of $H_0$ (the unique quasi-split group that is an inner form of $H$) and $\Sigma_q$ is a not-necessarily-constant \'etale group scheme.  Put $a_q$ for the image of $a$ under the twisting isomorphism $H^1(k, M_0) \to H^1(k, M_q)$.  By construction, $\iota(a_q) = 0$, so $a_q$ is the image of a 1-cocycle $b_q$ with values in $H_q(\kbar)$.

%%%%%%%%%%%%%%%%%%%%%%%%%%%%%%%%%

\section{$A_1$ case: $k$-forms and groups with a relative root system of type $BC_1$ } \label{SL2.eg}

We follow the notation of the preceding section, i.e., we consider an adjoint simple group $G$ with a minuscule $\SL_2$.  Such a group is obtained by twisting $G_0$ by a 1-cocycle $z$ with values in $M_0(\kbar)$.  We now describe concrete interpretations of the resulting form $H$ of the identity component of $M$ in terms of other algebraic structures, and indicate the correspondence between isotropy of $H$ (i.e., possible Tits indexes) and properties of that structure.  

We keep a specific focus on conditions for $H$ to be anisotropic, equivalently, for $G$ to have a relative root system of type $BC_1$.  Note that, if $H$ contains a split torus of rank one, then the quartic form $q$ must vanish on each of its non-trivial eigenspaces; that is, $q$ does not represent zero, then $H$ is anisotropic.  We prove the converse when $G$ has type $D_4$. \\

Consider first the case $G_0 = \PGL_{n+2}$.  If $G$ is inner, then $z$ is the image of some $z_0 \in H^1(k, \GL_n)$, so is trivial by Hilbert's theorem $90$. Hence we cannot find an inner twisting with $H$ anisotropic, and indeed a $G$ of inner type $A$ with a minuscule $\SL_2$ is split.

Suppose now that $G$ is \emph{not} inner, so it is an inner form of the quasi-split group $PU_{n+2}$ corresponding to a quadratic extension $K$ of $k$ as at the end of the preceding section, and $H_q$ is isomorphic to the unitary group $U_n$.  The set $H^1(k,U_n)$ classifies non-degenerate Hermitian spaces $W$ of rank $n$ over the quadratic field extension $K$, i.e., $H = U(h)$ for some such form.  The group $G$ will have a relative root system of type $BC_1$ over $k$ if and only if $H$ is anisotropic if and only if $h$ is anisotropic.  The group $G$ is the projective unitary group of the Hermitian space $W + N$, where $N$ is a split Hermitian space of dimension $2$. The decomposition of the Lie algebra over $k$ as in \eqref{sum.master} is
\[
\su(W+N) = \sl_2 \otimes 1 + 1 \otimes \fu(W) + V_2 \otimes W.
\]
Note that $\sl_2 \cong \su(N)$.\\

Looking at Table \ref{BC1.table} in the previous section, and using the fact that $H^1(k,\SL_n) = H^1(k,\Sp_{2n}) = 1$, we see that there are no groups $G$ with a relative root system of type $BC_1$ for inner forms of the split groups $B_n$, $C_n$, $G_2$, and $F_4$, as well as for inner forms of the split group $E_6$. 
However, the quasi-split groups $D_n$, $^2\!D_n$, $\tD$, $\sD$, $\dE$, $E_7$, and $E_8$ have inner forms over \emph{certain} fields $k$ with a relative root system of type $BC_1$. We can make this more explicit by studying certain algebraic structures on the representation $W$ of $H$.\\

\medskip

We now consider in some detail the case where $G$ has type $D_4$, i.e., where $q_0$ is Cayley's hyperdeterminant.  The group $\Sigma$ is a copy of the symmetric group on three letters.  As above, there is a natural map of $M_0 \to \Aut(W_0, \bilform_0, q_0)$, and by  \cite[Cor.~49]{Helenius} or \cite[Cor.~9.10]{BGL} the latter group is generated by the image of $M_0$ and $\mu_4$ acting as scalars.  By Galois descent, all twisted forms of the hyperdeterminant triple system are obtained by this construction.

For this case, we can prove the following.
\begin{prop} \label{hyperdet}
For a twisted form $q$ of the hyperdeterminant, the automorphism group of $q$ is isotropic if and only if $q$ represents zero.
\end{prop}

\begin{proof}
Over $\kbar$, $q$ is isomorphic to the hyperdeterminant $q_0$.  We leverage the study of the $H$-orbits in the projective variety $q_0 = 0$ as described in \cite{Roe:extra}, or see \cite{WeymanZ} for a more geometric viewpoint.  Specifically, there is a unique minimal closed $H_0$-invariant subvariety $X$, the $H_0$-orbit of $x_{\beta}$.  A smooth point of the variety $q_0 = 0$ is in the $H(\kbar)$-orbit of $v := x_{\alpha_1 + \alpha_2} + x_{\alpha_2 + \alpha_3} + x_{\alpha_2 + \alpha_4}$, where $x_\alpha$ denotes a generator for $\g_\alpha$ and we have numbered the simple roots $\alpha_1, \ldots, \alpha_4$ of $D_4$ as in \cite{Bou:g4} so that $\alpha_2$ corresponds to the central vertex of the Dynkin diagram.  Combining $q_0$ and $\bilform_0$, we find an $H_0$-invariant symmetric trilinear map $t_0 \!: W_0 \times W_0 \times W_0 \to W_0$ such that $\langle t_0(w,w,w),w \rangle = q_0(w)$ for all $w \in W_0$.  Because $v$ is a smooth point, $t_0(v,v,v) \ne 0$, and it follows from the $H_0$-invariance of $t_0$ that $t_0(v,v,v)$  is in the $k$-span of $x_{\alpha_2}$, i.e., belongs to $X$.  Looking now at $H$ and $q$ over $k$, if the variety $q = 0$ is nonempty, then we take $v$ to be a smooth point and observe that $X(k)$ contains $t(v,v,v)$ so is nonempty.  Then the stabilizer of $t(v,v,v)$ in $H$ is a parabolic subgroup and $H$ is isotropic.
\end{proof}

We can exhibit inner forms of quasi-split groups of type $D_4$ with a relative root system of type $BC_1$. Let $K$ be the cubic \'etale algebra determined by $\iota(a)$, so $H_q$ is the group $\Res_{K/k} \SL_2/ \Res_{K/k} (\mu_2)_{N=1}$. The inner forms of $H_q$ are isomorphic to $\Res_{K/k}(SL_1(Q))/\Res_{K/k} (\mu_2)_{N=1}$ for $Q$ a quaternion algebra with center $K$ such that the corestriction of $Q$ to $k$ (which is a central simple $k$-algebra of dimension $8^2$) is a matrix algebra.
(When $G_0$ is split, $K = k \times k \times k$ and $Q$ corresponds to three quaternion algebras $(Q_1,Q_2,Q_3)$ over $k$ such that the tensor product $Q_1 \otimes Q_2 \otimes Q_3$ is a matrix algebra.)  Explicitly, by \cite[43.9]{KMRT}, the quaternion algebra $Q$ has Hilbert symbol $(a, b)_K$ with $b \in k^\times$ and $a \in K^\times$ such that $N_{K/k}(a) = 1$.   When $K$ is a field, $H$ will be anisotropic if and only if $Q$ is a division algebra. When $K = k \times k \times k$, $H$ will be anisotropic if each of the quaternion algebras $Q_1$, $Q_2$, and $Q_3$ is a division algebra over $k$; in particular, the Brauer group of $k$ must contain a Klein $4$-group. It follows from Tits's Witt-type theorem that every isotropic group of type $\tD$ or $\sD$ arises in this way, see \cite{G:iso}.\\

Now suppose that $G_0$ is quasi-split of type $D_n$ for some $n \ge 5$; the Galois action on the Dynkin diagram determines the quadratic \'etale $k$-algebra $K$.  The group $H_q$ is isomorphic to $(\SL_2 \times \SO(q))/\mu_2$ for $q$ a sum of $n-2$ hyperbolic planes and the 2-dimensional orthogonal space $K$ with norm $N_{K/k}$.
Every inner twist $H$ of $H_q$ is of the form $(\SL_1(Q) \times \SO(h))/\mu_2$ for a (possibly split) quaternion $k$-algebra $Q$ and a skew-hermitian form $h$ on a $Q$-module $V$ of rank $n - 1$ such that $h$ has discriminant $K$ in the sense of \cite[\S10]{KMRT}.  Such a group $H$ is isotropic if and only if $h$ represents 0, i.e., if and only if there is some nonzero $v \in V$ such that $h(v,v) = 0$, see \cite[\S17.3]{Sp:LAG}. \\ %\cite[\S23.4]{Borel}.\\

Next suppose that $G_0$ is quasi-split of type $\dE$, which determines a quadratic field extension $K$ of $k$ and the quasi-split group $H_q$ is $\SU_6/\mu_3$.  Every inner form $H$ of $H_q$ is $\SU(B, \tau)/\mu_3$ for $B$ a central simple $K$-algebra of dimension $6^2$ and $\tau$ an involution on $B$ that  restricts to the nontrivial $k$-automorphism of $K$ and such that the discriminant algebra $D(B, \tau)$ defined in \cite[\S10.E]{KMRT} is split. Indeed, the Brauer class of the discriminant algebra is the Tits algebra for the representation $W$.  Because $B^{\otimes 3}$ is Brauer-equivalent to $D(B, \tau) \otimes K$ by \cite{Ti:R} or \cite[10.30]{KMRT}, it follows that $B = M_2(B_0)$ for some central simple $K$-algebra $B_0$ of dimension $3^2$ whose corestriction to $k$ is a matrix algebra.  Such a group $H$ is isotropic if and only if $\tau(b)b = 0$ for some nonzero $b$, i.e., if and only if $\tau$ is isotropic in the sense of \cite[6.3]{KMRT}.  Alternatively, one can view $\tau$ as the involution adjoint to a hermitian form $h$ on a rank 2 $B_0$-module $V$ as in \cite[\S4.A]{KMRT}, in which case we have: $H$ is isotropic if and only if $h(v,v) = 0$ for some nonzero $v \in V$.\\

When $G_0$ is split of type $E_7$, $H_0$ is a half-spin group and $W_0$ is the half-spin representation, i.e., $H_0$ is the image of $\Spin_{12} \to \GL(W_0)$.  Every inner form $H$ of $H_0$ is
 isogenous to $\SO(A, \sigma)$ where $A$ is a central simple $k$-algebra of dimension $12^2$ and $\sigma$ is an orthogonal involution with trivial discriminant such that the even Clifford algebra $C(A, \sigma)$ as defined in \cite[\S8]{KMRT} has one split component (namely the action on $W$).  Such pairs $(A, \sigma)$ have recently been described more explicitly, see \cite{QT12}.  Such an $H$ is isotropic if and only if the involution $\sigma$ is isotropic, i.e., if and only if $\sigma(a)a = 0$ for some nonzero $a \in A$.  
 
 \begin{rmks}[for $G$ of type $E_7$] See \cite[Prop.~3]{Igusa} for a description of the $H_0$-orbits on $W_0$.

To provide an anisotropic form $H$ of $H_0$, it is sufficient to produce an anisotropic 12-dimensional quadratic form in $I^3 k$ over some $k$.  This is easily done using Pfister's explicit description of such forms from \cite{Pfister}.
\end{rmks}

Finally when $G_0$ is split of type $E_8$, $W_0$ is the 56-dimensional minuscule representation of $H_0$, the split simply-connected group of type $E_7$.  Each inner form $H$ of $H_0$ has a corresponding $56$-dimensional representation $W$ over $k$ and we obtain then twisted forms of the Freudenthal triple system arising in the split case. As above, if $H$ is isotropic, then $q(w) = 0$ for some nonzero $w \in W$.  See \cite[\S7]{G:struct} for the structures in $W$ corresponding to parabolic subgroups of $H$ and, for example, \cite{Krut:E7} for a discussion of the variety $q = 0$ defined by the vanishing of the quartic form.

\begin{rmks}[for $G$ of type $E_8$]
(i): See \cite{Roe:extra}, \cite[Th.~7.6]{G:struct}, or \cite{Krut:E7} for a description of the $H$-orbits in $W$.  The description in \cite{G:struct} describes $k$-points on the projective homogeneous spaces for $H$ in terms of \emph{inner ideals} in $W$, i.e., subspaces $I$ such that $t(I,I,W) \subseteq I$.

(ii): Diverse constructions of anisotropic pure inner forms $H$ exist, see for example \cite[Prop.~2(B)]{Ti:si}, \cite[Example 7.2]{AF:CD}, \cite[Appendix A]{G:lens}, or \cite[Cor.~10.17]{GPS}.

(iii): Groups with relative root systems of type $BC_1$, viewed from the angle of Lie algebras with a 5-term grading as in \eqref{5term}, have been studied in the context of structurable algebras as in \cite{A:structintro} and  \cite{A:models}.

(iv): When $G$ has type $D_4$, we proved (Prop.~\ref{hyperdet}) that if $q$ represents zero, then $H$ is isotropic.  No proof on the same outline is possible in the $E_8$ case, as we illustrate with examples.  Specifically, first note that there are groups of type $E_8$ with semisimple anisotropic kernel of type $D_6$ or $E_6$.  For such groups, $H$ is isotropic with semisimple anisotropic kernel of the same type, and the corresponding form $q$ represents zero, i.e., there is a smooth $k$-point $v$ on the hypersurface $q = 0$.  The groups $H$ with anisotropic kernel of type $E_6$ correspond to a $W$ containing a 1-dimensional inner ideal but no 12-dimensional inner ideal; those with anisotropic kernel of type $D_6$ correspond to a $W$ containing a 12-dimensional inner ideal but no 1-dimensional inner ideal.  Therefore, there cannot be a deterministic mechanical procedure to construct from $v$ an inner ideal of $W$, in contrast to the $D_4$ case where the $k$-span of $t(v,v,v)$ provides a 1-dimensional inner ideal.

(v): Our methods do fail to capture four possibilities with relative root system of type $BC_1$, corresponding to $G$ having one of the following Tits indexes:
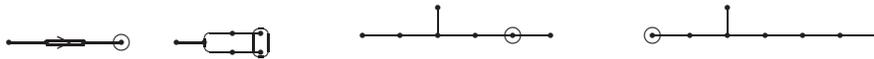
\[
\begin{picture}(4,1)
 \put(0,0){\usebox{\Fiv}} 
    \put(3.5,0.5){\circle{\lrad}}
\end{picture} \quad
\begin{picture}(3.5, 1.1)
\put(0,0){\usebox{\dEpic}}
\put(2.5,0.5){\oval(0.4,0.75)}
\end{picture}\quad
\begin{picture}(7,1.9)
 \put(0,0){\usebox{\Esev}} 
  \put(5,0.7){\circle{\lrad}}
\end{picture}\quad
\begin{picture}(8,1.9)
 \put(0,0){\usebox{\Eviii}} 
 \put(1,0.7){\circle{\lrad}}
 \end{picture}
\]
In these cases, the long roots have multiplicity 7, 8, 10, and 14 respectively.
\end{rmks}

\begin{rmk*}
The paper \cite{BDeMedts:new} gives results related to the case where $G$ has relative root system of type $BC_2$ and $J = \SL_2 \times \SL_2$.
\end{rmk*}

%%%%%%%%%%%%%%%%%%%%%%%%%%%%%%%%%%%%%%%%%%%
\section{Minuscule embeddings and relative root systems} \label{two.lengths}

Let $G$ be a split, simple adjoint algebraic group with roots of different lengths.  As mentioned in the introduction, the embedding $J \to G$ is minuscule when $J$ is the subgroup generated by the long root subgroups.  There are four cases to consider.\\

For type $B_n$, $G$ is the split adjoint group $\SO_{2n+1}$ and $J$ is the subgroup which fixes a non-isotropic line in the standard representation $V_{2n+1}$, with orthogonal complement $V_{2n}$. This gives an isomorphism of $J$ with the split even orthogonal group $\SO_{2n}$. The action of $J$ on $\g/\frak{j}$ is given by the standard representation $V_{2n}$.\\

For type $C_n$, $G$ is the adjoint group $\Sp_{2n}/\mu_2$ and $J$ is the subgroup stabilizing a decomposition of the symplectic space into non-degenerate planes. The group $J$ is isomorphic to the  split group $\SL_2^n/\Delta \mu_2$, and the action of $J$ on $\g/\frak{j}$ is by a direct sum of the four dimensional representations $V_2^i \otimes V_2^j$, with $1 \leq i < j \leq n$.\\

For type $G_2$,  the subgroup $J$ is isomorphic to $\SL_3$ and its action on $\frak{g_2}/\frak{\sl_3}$ is by the direct sum of the two three dimensional representations $V_3$ and $V_3^{\vee}$.\\

For type $F_4$, the subgroup $J$ is isomorphic to the split group $\Spin_{4,4}$ and its action on $\frak{f_4}/\frak{\spin_{4,4}}$ is by the direct sum of the three eight dimensional representations $V_8$, $V_8'$, and $V_8''$.\\

In all four cases, the centralizer of $J$ in $G$ is the center $Z(J)$, which is isomorphic to $\mu_2$, $(\mu_2)^{n-1}$, $\mu_3$, and $(\mu_2)^2$ respectively. The pinned outer automorphism group of $J$ is isomorphic to the symmetric group $S_k$ with $k = 2,n,2,3$ respectively, and the normalizer of $J$ in $G$ is equal to $J.S_k$. We should emphasize that in all these cases, we are only establishing the existence of a minuscule embedding, not the uniqueness up to conjugation in $G$ as we did for $\SL_2$. For example, the adjoint group $\PGL_3(k)$ acts on the set of minuscule embeddings $\SL_3 \rightarrow G_2$ over $k$ by its action by conjugation on $\SL_3$. Only the conjugates by the subgroup $\SL_3(k)/\mu_3(k)$ yield conjugate embeddings. Hence the conjugacy classes of embeddings form a principal homogeneous space for the quotient group $k^\times/k^{\times 3}$. In all the four cases, $J$ is isomorphic to the split group mentioned, but the isomorphism is not unique and $J$ has inner automorphisms which do not come from conjugation in $G$.\\

We now consider the case where $G$ need not be split over $k$, but has a relative root system of type $B_n$, $C_n$, $G_2$, or $F_4$.  Let $S$ be a maximal split torus in $G$ and \emph{assume that the long root spaces for $S$ acting on the Lie algebra $\g$ all have dimension one.}  Note that this hypothesis is automatic in case $G$ is split, because root spaces for a maximal torus all have dimension one.  It also holds when the relative root system is of type $G_2$ or $F_4$, as one can see by comparing the table of relative root systems from \cite[pp.~129--135]{Selbach} with Tables \ref{G2.table} and \ref{F4.table}.

Let $J \rightarrow G$ be the subgroup generated by $S$ and the long root groups. Then the subgroup $J$ is given above, and its action on $\g/\frak{j}$ decomposes as a direct sum of minuscule representations. Indeed, the remaining weights for $S$ are the short roots, and they are the weights which occur in the minuscule representations of $J$. These minuscule representations of $J$ will now occur with higher multiplicity in $\g/\frak{j}$, as the short root spaces will have multiplicity greater than one when $G$ is not split.

Let $H$ be the centralizer of $J$ in $G$. Since the ranks of $J$ and $G$ are the same (they both have maximal split torus $S$) the subgroup $H$ is anisotropic. It contains the anisotropic kernel of $G$ as its connected component, as the anisotropic kernel must act trivially on each long root space. Since $H$ centralizes the torus $S$, it acts linearly on each short root space $W_{\alpha} \subset \g$, and the isomorphism class of the representation $W_{\alpha}$ depends only on the orbit of the short root $\alpha$ under the action of the Weyl group of $S$ in $J$.\\

For $G$ with relative root system of type $B_n$ there is a single orbit of the Weyl group of $J = \SO_{2n}$ on the set of $2n$ short roots. The action of $J \times H$ on the quotient $\g/(\frak{j} + \h)$ is given by the tensor product $V_{2n} \otimes W$, where $W$ is the orthogonal representation of $H$ on the short root space $W_{\alpha}$ with $\alpha = e_1$.\\

For $G$ with relative root system of type $C_n$ there are $\binom{n}{2}$ orbits of the Weyl group of $J = (\SL_2)^n/\Delta \mu_2$ on the set of $4\binom{n}{2}$ short roots. The action of $J \times H$ on the quotient $\g/(\frak{j} + \h)$ is given by the direct sum of representations $\sum_{1 \leq i < j \leq n} (V_2^i \otimes V_2^j) \otimes W_{ij}$, where $W_{ij}$ is the orthogonal representation of $H$ on the short root space $W_{\alpha}$ with $\alpha = e_i + e_j$. Although these representations are not isomorphic, they are exchanged by the outer automorphism group of $J$, so all have the same dimension.\\

For $G$ with relative root system of type $G_2$ there are two orbits of the Weyl group of $J = \SL_3$ on the set of $6$ short roots. The action of $J \times H$ on the quotient $\g/(\frak{j} + \h)$ is given by the direct sum of representations $V_3 \otimes W + V_3^{\vee}\otimes W^{\vee}$, where $W$ is the representation on one of the short root spaces. We will see in \S\ref{SL3.eg} that $W$ and its dual $W^{\vee}$ have dimensions either 1, 3, 9, or $27$, and that both have an $H$-invariant cubic form.\\

For $G$ with relative root system of type $F_4$ there are three orbits of the Weyl group of $J = \Spin_{4,4}$ on the set of $24$ short roots. The action of $J \times H$ on the quotient $\g/(\frak{j} + \h)$ is given by the direct sum of representations $V_8 \otimes W + V_8'\otimes W' + V_8'' \otimes W''$, where $W$, $W'$, and $W''$ are three orthogonal representations of the same dimension. This dimension is either 1, 2, 4, or $8$, see \S\ref{D4.eg}.\\

Here is an example where the relative root system has type $B_n$ and the long root spaces have dimension one. Let $V$ be a non-degenerate orthogonal space over $k$ of odd dimension $d$ and rank $n$, so $d \geq 2n+1$. Let $X$ and $X'$ be a pair of dual maximal isotropic subspaces of dimension $n$, and let $W = X + X'$  be the corresponding non-degenerate subspace of dimension $2n$. Then $V = W + W^{\perp}$ and the adjoint group $G = \SO(V)$ has a relative root system of type $B_n$. The long roots have multiplicity one and give the subgroup $J = \SO(W) = \SO_{2n}$. The short roots have multiplicity equal to the dimension of $W^{\perp}$ and the centralizer $H = \O(W^{\perp})$ of $J$ acts on the short root spaces by the standard representation. The decomposition of the Lie algebra as in \eqref{sum.master} is 
\[
\so(V) = \so(W) + \so(W^{\perp}) + W \otimes W^{\perp}.
\]
A similar decomposition occurs for orthogonal spaces of even dimension $d \geq 2n+2$, where $n$ is the rank.\\

An example where the relative root system has type $C_n$ and the long root spaces have dimension one comes from the real groups $G$ that act on tube domains. Here $n$ is the rank of the domain. We will assume $n \geq 3$, as the cases where $n = 2$ are already covered by the $B_n$ case above. There are then three groups $G = \Sp_{2n}/\mu_2$,
$G = \PU_{n,n}$, and $G = \SO^*_{4n}/\mu_2$, together with the exceptional group $E_{7,3}/\mu_2$ which only occurs when $n = 3$. In the first case $G$ is split, $H$ is the center of $J$, and the orthogonal representations $W_{ij}$ all have dimension one. In the second case, $G$ is quasi-split, $H = \U_1^n/\U_1$, and the orthogonal representations $W_{ij}$ all have dimension $2$.
In the third case, $H = (\SU_2)^n/\Delta \mu_2$ and the orthogonal representations $W_{ij}$ all have dimension four. In the exceptional case, $H$ is the compact form $\Spin_8$ of $\Spin_{4,4}$ and the orthogonal representations $W,W'$ and $W''$ all have dimension $8$.

\section{$A_2$ case: Minuscule embeddings of $\SL_3$} \label{A2.basic}

In this section, our objective is to describe the minuscule embeddings of $\SL_3$ into split, simple  groups $G$ of adjoint type over $k$.  (We will use this description to give a classification of groups with a relative root system of type
 $G_2$.) If we have such an embedding, with centralizer $H$, we obtain a $\mu_3$-decomposition of the Lie algebra of $G$ as in \eqref{sum.master}:
 \[
 \g = \sl_3 + \h + V_3 \otimes V + V_3^{\vee} \otimes V^{\vee}.
 \]
  
% We will assume that the Lie bracket gives a surjective map of $\SL_3 \times H$ modules 
% $$\wedge^2(V_3 \otimes V) \rightarrow V_3^ {\vee} \otimes V^{\vee}.$$
% This occurs in all the cases where $G$ has a relative root system of type $G_2$, and eliminates certain trivial cases, like the embedding $\SL_3 \rightarrow \PGL_n$ stabilizing a subspace of dimension $3$, where $\sl_3 + \h + V_3 \otimes V$ is a parabolic subalgebra of $\g$, and the Lie bracket is trivial on the abelian nilpotent radical $V_3 \otimes V$.\todo{Where do we make use of this discussion?  What is the point we are trying to make?}  \\
 Restricting the minuscule embedding $\SL_3 \rightarrow G$ to an embedded $\SL_2 \hookrightarrow \SL_3$ that is itself minuscule provides a minuscule embedding $\SL_2 \hookrightarrow G$. Indeed, the restriction of 
the standard representation $V_3$ (and its dual) is the direct sum of the standard representation of $\SL_2$ and the trivial representation, and the restriction of the adjoint representation $\sl_3$ is the direct sum of the adjoint representation $\sl_2$, two copies of the standard representation and one copy of the trivial representation. Hence the decomposition of $\g$ under $\SL_2$ is as in \eqref{sum.sl2}:
\[
\g = \sl_2 + \m + V_2 \otimes W.
\]
Let $S$ be the centralizer of $\SL_2$ in $\SL_3$, which is a split torus of dimension one and has character group isomorphic to $\Z$. We can fix an isomorphism with $\Gm$, so that the characters of $S$ on $V_3$ are $1,1, -2$ and the characters of $S$ on $V_3^{\vee}$ are $-1,-1,2$. It follows that the characters of $S$ on $\sl_3 \subset V_3 \otimes V_3^\vee$ are $3,3,0,0,0,0,-3,-3$. The centralizer of $S$ in $\SL_3$ is isomorphic to $\GL_2$. Since $S$ centralizes $\SL_2$, it is contained in the stabilizer $M$ of the minuscule embedding of $\SL_2$ and acts on the two representations $\m$ and $W$. \\

From the decomposition of $\g$ into representations of $\SL_3 \times H$
we see that the only characters of $S$ that appear (with multiplicities) in $\g$ are $\{-3,-2,-1,0,1,2,3\}$. Since the intersection of $S$ and $\SL_2$ is the center $\mu_2$, the torus $S$ acts by even characters on $\m$ and by odd characters on $W$. 
Therefore
$S$ acts by the three characters $-2,0,2$ on $\m$,  and by the four characters $-3,-1,1,3$ on $W$. The characters $3$ and $-3$ only appear in the summand $\sl_3$, so each appears with multiplicity $2$ in $\g$. Hence the characters $3$ and $-3$ each appear in $W$ with multiplicity one, and the characters $1$ and $-1$ each appear with multiplicity equal to $\dim V$. The multiplicities of the characters $2$ and $-2$ in the representation of $S$ on $\m$ are also equal to $\dim V$. By counting dimensions, this gives the multiplicity of the trivial character of $S$ in $\m$, and we see that the centralizer of $S$ in $M$ is isomorphic to $S.H$. \\

Since $S$ has only three weights on $\m$, the inclusion $S \rightarrow M$ is a minuscule co-character. The centralizer $S.H$ of $S$ is a Levi subgroup of a parabolic in $M$, whose abelian unipotent radical is isomorphic to the representation $V$ of $H$. Moreover, the action of this centralizer of $S$ on the symplectic representation $W$ of $M$ stabilizes two lines, where $S$ acts by the characters $3$ and $-3$ and $H$ acts trivially. Since we know $M$ and the representations $\m$ and $W$ from our previous classification of minuscule embeddings of $\SL_2$, these conditions determine the co-character $\mu: S \rightarrow G$ up to conjugacy (when it exists). If we fix a maximal split torus $T$ in $G$ and a set of positive roots $\alpha_i$, there is no loss of generality in assuming that $\mu$ is a dominant co-character of $T$,  and we can describe $\mu$ by giving the inner products $\langle \mu, \alpha_i \rangle$ for all $i$. \\

When $G$ is exceptional, the co-character $\mu$ has inner product $1$ with a unique simple root $\alpha$, which has multiplicity $3$ in the highest root, and inner product $0$ with all other simple roots. When $G$ has type $D_4$, $\mu$ has inner product $1$ with the three simple roots $\alpha_i$ which have multiplicity $1$ in the highest root, and inner product $0$ with the remaining simple root. When $G$ has type $B_n$ or $D_n$ for $n \geq 5$, $\mu$ has inner product $1$ with the two simple roots $\alpha_1$ and $\alpha_3$ and inner product $0$ with the remaining simple roots.  See Figure \ref{SL3.fig} for an illustration.  (When $G$ has type $A_{n-1}$, the minuscule $\SL_3$ is the subgroup of  $\PGL_n$ stabilizing a subspace of dimension $3$, and $\sl_3 + \h + V_3 \otimes V$ is a parabolic subalgebra of $\g$.  We ignore this degenerate case, cf.~\cite[Th.~4.8]{Vinberg:na}.) \\

\begin{figure}[hbt]
\begin{center}
%%%%%%%%%%%%%%%%%%%%%%
\begin{picture}(7,2)(0,0)
\multiput(1,1)(1,0){4}{\circle*{\darkrad}}
\put(1,1){\line(1,0){3}}
\put(4,0.8){\hspace{6pt}\mbox{$\cdots$}}
\multiput(6.2,1)(1,0){3}{\circle*{\darkrad}}
\put(6.2,1){\line(1,0){1}}
\put(7.2,0.9){\line(1,0){1}}
\put(7.2,1.1){\line(1,0){1}}

\put(7.4, 0.85){\makebox(0.2,0.3)[s]{$>$}}
\multiput(1,1)(2,0){2}{\circle{\lrad}}
\end{picture} 

%%%%
\begin{picture}(7,2)(0,0)
\multiput(1,1)(1,0){4}{\circle*{\darkrad}}
\put(1,1){\line(1,0){3}}
\put(4.2,0.8){\hspace{6pt}\mbox{$\cdots$}}
\multiput(6.2,1)(1,0){2}{\circle*{\darkrad}}
\put(6.2,1){\line(1,0){1}}
\put(7.2,1){\line(4,3){1}}
\put(7.2,1){\line(4,-3){1}}
\put(8.2,1.75){\circle*{\darkrad}}
\put(8.2,0.25){\circle*{\darkrad}}
\multiput(1,1)(2,0){2}{\circle{\lrad}}
\end{picture} \hspace{1in}
%%%%
\begin{picture}(3,2)(0,0)
\multiput(0.2,1)(1,0){2}{\circle*{\darkrad}}
\put(0.2,1){\line(1,0){1}}
\put(1.2,1){\line(4,3){1}}
\put(1.2,1){\line(4,-3){1}}
\put(2.2,1.75){\circle*{\darkrad}}
\put(2.2,0.25){\circle*{\darkrad}}
\put(0.2,1){\circle{\lrad}}
\put(2.2,1.75){\circle{\lrad}}
\put(2.2,0.25){\circle{\lrad}}
\end{picture}

%%%%%%%%%%%%%%%%%%%%%%
\begin{picture}(5,2)
  \put(0,0){\usebox{\Evi}} 
  \put(2.5,0.5){\circle{\lrad}}
\end{picture}
%%%%
\begin{picture}(7,1.9)
  \put(0,0){\usebox{\Esev}} 
  \put(2,0.7){\circle{\lrad}}
\end{picture}
%%%%
\begin{picture}(8,1.9)
 \put(0,0){\usebox{\Eviii}} 
 \put(6,0.7){\circle{\lrad}} 
% \put(7,0.7){\circle{\lrad}}
 \end{picture}  \\
%%%%
\begin{picture}(4,1)
 \put(0,0){\usebox{\Fiv}} 
    \put(1.5,0.5){\circle{\lrad}}
\end{picture}
%%%%
\begin{picture}(4,1)
\put(0,0){\usebox{\Gii}}
\put(2,0.5){\circle{\lrad}}
\end{picture}
 \caption{Dynkin diagrams with circles around those simple roots $\alpha$ such that $\langle \mu, \alpha \rangle \ne 0$.} \label{SL3.fig}
\end{center}
\end{figure}

Having determined the co-character $\mu$, we obtain a seven term grading of $\g$
\begin{equation} \label{SL3.seven}
\g = \g(-3) + \g(-2) + \g(-1) + \g(0) + \g(1) + \g(2) + \g(3)
\end{equation}
where the summands are representations of the centralizer of $S$, which is isomorphic to $(\GL_2 \times H)/\mu_3$. The summand $\g(1)$ is isomorphic to the representation $V_2 \otimes V$, the summand $\g(2)$ is isomorphic to $\det \otimes V^{\vee}$, and the summand $\g(3)$ is isomorphic to $V_2 \otimes \det$ as a representation of $\GL_2$ (tensor the trivial representation of  $H$). The centralizer of $\mu_3 \hookrightarrow S$ is then isomomorphic to $(\SL_3 \times H)/\mu_3$, and this gives a minuscule embedding of $\SL_3$. We describe the centralizer $H$
and the representations $V$ and $V^{\vee}$ of $H$ in Table  \ref{G2.table}, cf.~\cite{DeligneGross} and \cite[Table 2]{Vinberg:na}. \\

\begin{table}[hbt]
\[
\begin{array}{ccc}
G & H & V \\ \hline
\SO_{2n+5} & \SO_{2n-1} \times \GL_1 & V_1(-2) \oplus V_{2n-1}(1)\\
\SO_{2n+4}/\mu_2 &(\SO_{2n-2} \times \GL_1)/\mu_2 & V_1(-2) \oplus V_{2n-2}(1)\\
G_2 &\mu_3& V_1\\
D_4 & (\GL_1^3)_{N=1} & V_1 + V_1' + V_1''\\
F_4  &\SL_3  & \Sym^2(V_3) = V_6\\
E_6/\mu_3 &(\SL_3 \times \SL_3)/\mu_3  &V_3 \otimes V_3' = V_9  \\
E_7/\mu_2 &\SL_6/\mu_2  &\wedge^2(V_6) = V_{15}  \\
E_8  & E_6 &V_{27}  \\
\end{array}
\]
\caption{For minuscule $\SL_3$ in $G$, the centralizer $H$ and its representation $V$} \label{G2.table}
\end{table}

%\newpage

Since the relevant node or nodes on the Dynkin diagram are stable under graph automorphisms, we find that the full stabilizer $M$ of the minuscule embedding $\SL_3 \rightarrow G$ in $\Aut(G)$ is $O_{2n-2} \times \GL_1$ for type $D_{n+2}$, $(\GL_1^3)_{N=1}.S_3$ for type $D_4$, and $(\SL_3 \times \SL_3)/\mu_3 .2$ for type $E_6$.

\section{$A_2$ case: invariant tensors} \label{SL3.invt}

We now approximately follow the path of \S\ref{SL2.invt}, except with a minuscule embedding $\SL_3 \to G$ as in the previous section.  We assume here that $G$ is of type $F_4$ or $E$.  (The tiny cases where $G$ has type $G_2$ or $D_4$ have similar outcomes but involve ad hoc arguments that we omit here.)

Let $G'$ be the subgroup of $G$ generated by $H$ and the root subgroups $G_{\pm \alpha}$ where $\alpha$ is the unique simple root such that $\langle \mu, \alpha \rangle \ne 0$ as in Figure \ref{SL3.fig}.  It is semisimple.  The coefficient of $\alpha$ in the highest root of $G'$ is 1, so $\mu$ gives a 3-grading $\g' = \g'(-1) \oplus \g'(0) \oplus \g'(1)$ such that $\g'(1) = V$ and $\g'(-1)$ is the dual of $V$ as a representation of $H$.  (This can be seen by exactly the same deduction as the observation that $\g(1) = V_2 \otimes V$ in \eqref{SL3.seven}, appealing to \cite[Th.~2]{ABS}.)  The subalgebra $\g'$ is called the \emph{stock} in \cite{Vinberg:na}.

By the same argument as in \S\ref{SL2.invt}, $k[V]^H = k[f]$ for some homogeneous $f$.  In case $k = \C$, a routine calculation with weights shows that $f$ has degree 3.  As in \cite[pp.~4767, 4768]{BGL}, one deduces that $\deg f = 3$ in all cases.  (The argument in \cite{BGL} is uniform and relies on \cite{RRS}.  Alternatively, one can calculate by hand in each case.)

Looking from a different angle, the 3-grading shows that $\g'(1) \oplus \g'(-1)$ is a \emph{Jordan pair}, meaning that the quadratic maps $Q_\eps \!: \g'(\eps) \to \Hom_k(\g'(-\eps), \g'(\eps))$ defined by
\[
Q_\eps(x)(y) := (\ad x)^2 y \quad \text{for $x \in \g'(\eps)$ and $y \in \g'(-\eps)$}
\]
for $\eps = \pm 1$
satisfy certain identities; see \cite{Loos} for an extensive theory.  This is the point of view of \cite[esp.~\S2]{Sp:jord}, \cite{Loos:alg}, \cite{Neher:3lie}, and \cite[Ch.~11]{Lopez}; it can be viewed in the context of the Tits-Kantor-Koecher construction of Lie algebras.  (Yet another angle is pursued in \cite{BenkartElduque}, where the authors allow the representation $\g/\sl_3$ of $\SL_3$ to include also copies of $\sl_3$ in addition to copies of $V_3$ and $V^\vee_3$, and use this to construct a structurable algebra from $\g$.)

Given a Jordan algebra $J$, one can construct from it a Jordan pair $(J, J)$, and the Jordan pair $\g'(1) \oplus \g'(-1)$ is of this form, see \cite[\S14, esp.~14.31]{Sp:jord} or \cite[Prop.~4.2]{Vinberg:na}.  In each case $J$ is a cubic Jordan algebra.  Specifically:
\begin{itemize}
\item For $G$ of type $E_8$, $G'$ is of type $E_7$ and $J$ is a 27-dimensional exceptional Jordan algebra, sometimes called an \emph{Albert algebra}.
\item For $G$ of type $E_7$, $G'$ is of type $D_6$ and $J$ is the Jordan algebra of 6-by-6 alternating matrices with norm the Pfaffian, as in \cite[14.19]{Sp:jord}.
\item For $G$ of type $E_6$, $G'$ is of type $A_5$ and $J$ is the Jordan algebra of 3-by-3 matrices with norm the determinant, as in \cite[14.16]{Sp:jord}.
\item For $G$ of type $F_4$, $G'$ is of type $C_3$ and $J$ is the Jordan algebra of 3-by-3 symmetric matrices with norm the determinant, as in \cite[14.17]{Sp:jord}.
\end{itemize}
Alternatively, $J$ is the Jordan algebra of 3-by-3 hermitian matrices with entries in a composition algebra $C$ of dimension 8, 4, 2, or 1 respectively.  

\section{$A_2$ case: $k$-forms and groups with relative root system of type $G_2$} \label{SL3.eg} 

We now describe $k$-forms of the groups appearing in the previous section.  As in \S\ref{SL2.eg}, we put a subscript 0 on the groups involved to indicate the split group.

The automorphism group $H'_0$ of the Jordan algebra structure on $V_0$ is the subgroup of $H_0$ fixing the identity element $e \in V$, see \cite[14.11]{Sp:jord} or \cite[Th.~4]{Jac:J1}.  Moreover, $H_0$ has a central $\mu_3$ that acts as scalars on $V$.  It follows that the stabilizer of the line $ke$ in $\mathbb{P}(V_0)$ is $\mu_3 \times H'_0$.  On the other hand, the $H_0$-orbit of $ke$ is dense; it is the collection of lines $kv$ such that $f(v) \ne 0$.  Therefore, the natural map $H^1(k, \mu_3 \times H'_0) \to H^1(k, H_0)$ is surjective as in \cite[9.11]{G:lens} (cf.~\cite{Sp:cubic}), and twisting $G_0$ by a cocycle with values in $H_0$ amounts to twisting separately by a cocycle with values in $H'_0$ and by a cocycle with values in $\mu_3$.  The latter twist does not affect the isomorphism class of the resulting $H$ and therefore by Tits's Witt-type theorem does not affect the isomorphism class of the resulting twist $G$ of $G_0$.  In summary, the twists of $G_0$ by a cocycle with values in $H_0$ can be obtained by twists by cocycles with values in $H'_0$.  In particular, the twist $V$ of $V_0$ so obtained will be a Jordan algebra, and the generic norm on $V$ will be a cubic form $f$ invariant under $H$.

Thus we can determine the groups $G$ with a relative root system of type $G_2$ such that the long roots have multiplicity one.  We use a method similar to our determination of the groups with a relative root system of type $BC_1$. Namely, the split torus in $G$ together with the long root groups generate a minuscule $\SL_3 \rightarrow G$. Let $G_0$ be the split inner form of $G$ over $k$, and let $\SL_3 \rightarrow G$ be a minuscule embedding as described in the previous section, associated to the co-character $\mu$. Let $M_0$ be the stabilizer of this embedding in $\Aut(G_0)$, so $M_0$ has connected component the group $H_0$ tabulated in Table \ref{G2.table}. Since all minuscule embeddings of $\SL_3$ into $G_0$ are conjugate over $\overline{k}$ we may choose an isomorphism $\phi: G_0 \rightarrow G$ over $\overline{k}$ which is the identity on the embedded subgroups $\SL_3$. This gives a cohomology class in $H^1(k,M_0(\overline{k}))$, which determines the isomorphism class of $H$ and $G$. The question is whether we can find such a class so that the corresponding form $H$ of $H_0$ is anisotropic.

For these Jordan algebras, the following are equivalent by \cite[37.12, 38.3]{KMRT}:
\begin{enumerate}
\item  \label{alb.cubic} The cubic form $f$ on $V$ represents zero.
\item \label{alb.zdiv} The algebra $V$ has zero divisors.
\item \label{alb.iso} $H$ is isotropic.
\end{enumerate}
That is, $G$ will have relative root system of type $G_2$ if and only if $f$ does not represent zero, if and only if $V$ is a division algebra.

In the cases where $\dim V = 6$ or 15 (i.e., $G_0 = F_4$ or $E_7$), the equivalent conditions hold.  This can be seen by Jordan algebra methods, as is done in \cite[37.12]{KMRT}.  It can be seen also by Galois descent, because inner twists of $\SL_3$ and $\SL_6/\mu_2$ are isotropic.

In the remaining cases, anisotropic forms of $H$ exist.  Specifically, for the case $\dim V = 9$, find a field $k$ and a central associative division $k$-algebra $A$ of dimension $3^2$.  The algebra $V$ with underlying vector space $A$ and product $a \cdot b = \frac12 (ab + ba)$ where juxtaposition denotes the associative multiplication in $A$ is a Jordan algebra without zero divisors.  Moreover, adjoining an indeterminate $t$ to $k$, there is an Albert algebra over $k(t)$ with no zero divisors, namely the ``first Tits construction'' denoted $J(A, t)$, compare \cite[Prop.~3(B)]{Ti:si}.

\begin{rmk*}
In case $G_0 = E_6$, the group $M_0$ with identity component $H_0$ has two components, and the same reasoning applies for twisting by a cocycle with values in $M_0$.  Twisting $V_0$ by a 1-cocycle with values in $M_0$ whose image in $H^1(k, M_0/H_0)$ is a quadratic field extension $K$ of $k$ gives a Jordan algebra with 
underlying vector space the $\tau$-symmetric elements of $(B, \tau)$ where $B$ is a central simple $K$-algebra and $\tau$ is a unitary involution on $B$ whose restriction to $K$ is the nontrivial $k$-algebra automorphism.
\end{rmk*}

%%%%%%%%%%%%%%%%%%%%%%%%%%%%%%
\section{$D_4$ case: minuscule embeddings of $\Spin_{4,4}$ and groups with relative root system of type $F_4$} \label{D4.eg}

Like the case of $G_2$, the groups with a relative root system of type $F_4$ are all exceptional. We obtain a minuscule embeddings of the long root subgroup $\Spin_{4,4} \rightarrow G$, which in the split cases gives the following decomposition of $\g$ as in \eqref{sum.master}:
\begin{equation} \label{f4.decomp}
\g = \spin_{4,4} + \h + V_8 \otimes W + V_8' \otimes W' + V_8'' \otimes W''.
\end{equation}
Table \ref{F4.table} lists the centralizers $H$ and the dimensions of the three orthogonal representations $W$, $W'$, and $W''$ of $H$ which occur, cf.~\cite{DeligneGross}.  
Note that there is a copy of the symmetric group $\Sigma$ on 3 letters in $G$ normalizing $J$ and acting as outer automorphisms on $J$.  (This is true in the case where $G = F_4$ as in \cite[23.13, 26.5, 38.7]{KMRT}, and the other embeddings $J \to G$ factor through an $F_4$ subgroup.)  So $\Sigma$ acts on $J.H$ and permutes the three $V_8 \otimes W$ summands; in this sense the three summands are interchangeable.  

\begin{table}[bht]
\[
\begin{array}{ccc}
G&H&\dim W \\ \hline
F_4 &\mu_2 \times \mu_2 &1\\
E_6/\mu_3 &(\mathbb {G}_m)^2 &2\\
E_7/\mu_2 &(\SL_2)^3/\Delta \mu_2 &4 \\
E_8 &\Spin_{4,4} &8 
\end{array}
\]
\caption{For minuscule $\Spin_{4,4}$ in $G$, the centralizer $H$ and its representation $W$}\label{F4.table}
\end{table}

We omit a ``top down'' analysis reconstructing the algebraic structure on $W$, although it is natural to think of it as a symmetric composition algebra as defined in \cite[\S34]{KMRT}.  

Alternatively, the additive decomposition \eqref{f4.decomp} is familiar from the theory of structurable algebras as in \cite[p.~1869, (c)]{A:models}, which takes a tensor product $C_1 \otimes C_2$ with $C_1$ an octonion algebra and $C_2$ any composition algebra and constructs from it a Lie algebra $\g$ with the same decomposition \eqref{f4.decomp}.  For a different view, see \cite{Eld:magic1}, \cite{Eld:magic2}.  In such ways, one can reconstruct Table \ref{F4.table} ``from the ground up''.

In the non-split case, the group $H$ will be anisotropic if and only if the quadratic norm form on $W$ does not represent zero over $k$, or equivalently, when the composition algebra is a division algebra. This will occur for the split group $F_4$, the quasi-split group $\dE$, as well as certain inner forms of $E_7$ and $E_8$. In these cases, the short root spaces have dimension $1$, $2$, $4$, and $8$ respectively as in Table \ref{F4.table}.

\bibliographystyle{elsarticle-num-names}
\bibliography{skip_master}

\end{document}